\pdfoutput=1
\RequirePackage{ifpdf}
\ifpdf 
\documentclass[pdftex]{sigma}
\else
\documentclass{sigma}
\fi

\usepackage{bm}

\numberwithin{equation}{section}
\newtheorem{Theorem}{Theorem}[section]
\newtheorem{Corollary}[Theorem]{Corollary}
\newtheorem{Lemma}[Theorem]{Lemma}
\newtheorem{Proposition}[Theorem]{Proposition}
{ \theoremstyle{definition}
\newtheorem{Definition}[Theorem]{Definition}
\newtheorem{Remark}[Theorem]{Remark} }

\newcommand{\pair}[2]{\langle #1, #2 \rangle}
\newcommand{\ignore}[1]{}

\newcommand{\ol}[1]{\overline{#1}}

\newcommand{\tn}[1]{\textnormal{#1}}
\renewcommand{\i}{{\mathrm{i}}}

\def\Ad{\ensuremath{\textnormal{Ad}}}
\def\ad{\ensuremath{\textnormal{ad}}}
\def\g{\ensuremath{\mathfrak{g}}}

\def\z{\ensuremath{\mathfrak{z}}}
\def\t{\ensuremath{\mathfrak{t}}}
\def\n{\ensuremath{\mathfrak{n}}}
\def\a{\ensuremath{\mathfrak{a}}}
\def\h{\ensuremath{\mathfrak{h}}}
\def\p{\ensuremath{\mathfrak{p}}}

\def\L{\ensuremath{\mathcal{L}}}

\def\R{\ensuremath{\mathcal{R}}}

\def\bC{\ensuremath{\mathbb{C}}}
\def\bR{\ensuremath{\mathbb{R}}}
\def\bZ{\ensuremath{\mathbb{Z}}}

\def\bQ{\ensuremath{\mathbb{Q}}}

\def\ker{\ensuremath{\textnormal{ker}}}
\def\supp{\ensuremath{\textnormal{supp}}}

\def\pr{\ensuremath{\textnormal{pr}}}
\def\dim{\ensuremath{\textnormal{dim}}}

\def\pt{\ensuremath{\textnormal{pt}}}

\def\index{\ensuremath{\textnormal{index}}}

\def\Ch{\ensuremath{\textnormal{Ch}}}
\def\Td{\ensuremath{\textnormal{Td}}}

\def\calDC{\ensuremath{\mathcal{D}_{\bC}}}

\def\supp{\ensuremath{\textnormal{supp}}}
\def\ann{\ensuremath{\tn{ann}}}
\def\de{\ensuremath{\tn{det}_\bC}}

\begin{document}

\newcommand{\arXivNumber}{1907.06113}

\renewcommand{\PaperNumber}{090}

\FirstPageHeading

\ShortArticleName{Quasi-Polynomials and the Singular $[Q,R]=0$ Theorem}

\ArticleName{Quasi-Polynomials and the Singular $\boldsymbol{[Q,R]=0}$\\ Theorem}

\Author{Yiannis LOIZIDES}

\AuthorNameForHeading{Y.~Loizides}

\Address{Pennsylvania State University, USA}
\Email{\href{mailto:yxl649@psu.edu}{yxl649@psu.edu}}

\ArticleDates{Received July 16, 2019, in final form November 13, 2019; Published online November 18, 2019}

\Abstract{In this short note we revisit the `shift-desingularization' version of the $[Q,R]=0$ theorem for possibly singular symplectic quotients. We take as starting point an elegant proof due to Szenes--Vergne of the quasi-polynomial behavior of the multiplicity as a function of the tensor power of the prequantum line bundle. We use the Berline--Vergne index formula and the stationary phase expansion to compute the quasi-polynomial, adapting an early approach of Meinrenken.}


\Classification{53D20; 53D50}

\vspace{-2mm}

\section{Introduction}
Let $(M,\omega)$ be a compact connected symplectic manifold equipped with an action of a compact connected Lie group $G$ by symplectomorphisms. Suppose that the action of $G$ is Hamiltonian, meaning that there is a $G$-equivariant map, the moment map,
\[ \mu_\g \colon \ M \rightarrow \g^\ast,\]
where $\g^\ast$ is the dual of the Lie algebra $\g=\tn{Lie}(G)$, satisfying the moment map condition
\begin{gather}\label{eqn:mm}
\iota(X_M)\omega=-{\rm d}\pair{\mu_\g}{X}, \qquad X \in \g.
\end{gather}
Let $\big(L,\nabla^L\big)$ be a $G$-equivariant prequantum line bundle with connection on $M$, i.e., $L$~is a~$G$-equivariant Hermitian line bundle with compatible connection $\nabla^L$, $\big(\nabla^L\big)^2=-2\pi \i \omega$ and the derivative of the $G$-action on $L$ satisfies Kostant's condition
\[ \L^L_X-\nabla^L_{X_M}=2\pi \i\pair{\mu_\g}{X}.\]
Choose a compatible almost complex structure $J$ on $M$, i.e., $\omega(Jw,Jv)=\omega(w,v)$ and $\omega(w,Jv)\allowbreak =:g(w,v)$ is a Riemannian metric. Let $D_L$ denote the Dolbeault--Dirac operator twisted by $\big(L,\nabla^L\big)$, an elliptic differential operator acting on sections of the spinor bundle $\wedge T^\ast_{0,1}M\otimes L$. The kernel of~$D_L$ carries an action of~$G$, and the $G$-equivariant index is defined to be the difference $\index_G(D_L):=\ker\big(D_L^{\tn{even}}\big)-\ker\big(D_L^{\tn{odd}}\big)$ of the kernel of~$D_L$ on even/odd degree forms, regarded as an element of the representation ring~$R(G)$.

The quantization-commutes-with-reduction theorem ($[Q,R]=0$ theorem) describes the multiplicity of the trivial representation in $\index_G(D_L)$ in terms of the symplectic quotient $M^{\tn{red}}:=\mu_\g^{-1}(0)/G$. When $0$ is a regular value of $\mu_\g$, $M^{\tn{red}}$ is an orbifold and the theorem states that $\index_G(D)^G$ equals the index of the twisted Dolbeault--Dirac operator $D^{\tn{red}}_{L^{\tn{red}}}$ on $M^{\tn{red}}$. The theorem was first conjectured by Guillemin--Sternberg~\cite{GuilleminSternbergConjecture}, and the general case ($M$, $G$ both compact, $0$ a regular value) was first proved by Meinrenken~\cite{MeinrenkenSymplecticSurgery}. Different proofs of the $[Q,R]=0$ theorem were given by Tian--Zhang~\cite{TianZhang} and Paradan~\cite{ParadanRiemannRoch}. The theorem has since been extended in various directions.

There are versions of the $[Q,R]=0$ theorem when $0$ is not necessarily a regular value, due to Meinrenken--Sjamaar~\cite{MeinrenkenSjamaar}; below we will give a precise statement of one of these results, involving a partial \emph{shift desingularization}, i.e., $\index_G(D_L)^G$ is related to the index on the symplectic quotient at a nearby weakly regular value. At the same time, we introduce some notation that will be of use later on.

Fix a maximal torus $T$ with Lie algebra $\t$. Let $\Lambda \subset \t^\ast$ be the (real) weight lattice. Given $\lambda \in \Lambda$, the corresponding character $T \rightarrow U(1)$ is written $t \mapsto t^\lambda={\rm e}^{2\pi \i\pair{\lambda}{X}}$ where $t={\rm e}^X$, $X \in \t$. Let $\R \subset \Lambda$ be the set of roots. We also fix a closed positive Weyl chamber $\t_+$, which determines a set of positive (resp.\ negative) roots $\R_{\pm}$. For each relatively open face $\sigma \subset \t_+^\ast$, the stabilizer~$G_\xi$ of points~$\xi \in \sigma$ under the coadjoint action, does not depend on~$\xi$, and will be denoted $G_\sigma$. If $\sigma_1 \subset \ol{\sigma}_2$ then $G_{\sigma_1}\supset G_{\sigma_2}$. Note also that $G_\sigma$ is connected and contains the maximal torus~$T$. The Lie algebra~$\g_\sigma$ decomposes into its semi-simple and central parts
$ \g_\sigma=[\g_\sigma,\g_\sigma]\oplus \z_\sigma$.
The subspace $\z_\sigma^\ast \subset \t^\ast$ is defined to be the annihilator of $[\g_\sigma,\g_\sigma]$, or equivalently the fixed point set of the coadjoint $G_\sigma$ action. The face $\sigma$ is an open subset of $\z_\sigma^\ast$.

Let $\Delta=\mu_\g(M)\cap \t_+^\ast$ be the moment polytope. A well-known theorem in symplectic geometry states that there is a~unique face $\sigma \subset \t^\ast_+$ of minimal dimension such that $\Delta \subset \ol{\sigma}$ (briefly, this is a~consequence of~\eqref{eqn:mm}, which implies that ${\rm d}\mu_\g$ has constant rank on the top dimensional $G$-orbit type stratum, and the complement of the latter has codimension at least~$2$); $\sigma$ is called the \emph{principal face} or \emph{principal wall}. The corresponding symplectic cross-section, called the \emph{principal cross-section},
$Y=\mu_\g^{-1}(\sigma)$
is a Hamiltonian $G_\sigma$-space. Moreover the semi-simple part $[G_\sigma,G_\sigma]$ of $G_\sigma$ acts trivially on $Y$. For further details, see for example~\cite{ConvexitySympCuts} and references therein.

Let $I \subset \z_\sigma^\ast$ be the smallest affine subspace containing $\Delta$. Let $\t_I\subset \t$ be the annihilator of the subspace parallel to $I$, and let $T_I=\exp(\t_I)\subset T$ be the corresponding subtorus. By equation~\eqref{eqn:mm}, $\t_I$ is the generic infinitesimal stabilizer of $Y$. In particular $T_I$ acts trivially, hence the quotient torus $T/T_I$ acts on $Y$. The moment map $\mu_\g$ may have no non-trivial regular values. But the restriction
\[\mu_\g|_Y\colon \ Y \rightarrow I\]
viewed as a map with codomain $I$, always has non-trivial regular values, and we will refer to these as \emph{weakly-regular values}. If $\xi$ is a weakly-regular value, then the reduced space $M_\xi=\mu_\g^{-1}(\xi)/G_\sigma$ is an orbifold. Let $L_\xi=L|_{\mu_\g^{-1}(\xi)}/G_\sigma$ be the corresponding (orbifold) line bundle over $M_\xi$.

\begin{Theorem}[\cite{MeinrenkenSjamaar}, see also \cite{ParadanRiemannRoch,WittenNonAbelian}]
\label{thm:MeinSja}
Let $(M,\omega,\mu_\g)$ be a compact connected Hamiltonian $G$-space with moment polytope $\Delta$. If $0 \notin \Delta$ then $\index_G(D_L)^G=0$. Otherwise for every weakly-regular value $\xi \in \Delta$ sufficiently close to $0$, $\index_G(D_L)^G$ equals the index of the Dolbeault--Dirac operator $D^{\tn{red}}_{L_\xi}$ on the reduced space $M_\xi$.
\end{Theorem}

We will now describe the main result of this article and its relation to Theorem \ref{thm:MeinSja}. Consider tensor powers $L^k$, $k \in \bZ_{>0}$ of the prequantum line bundle. For a dominant weight $\lambda$, let $\chi_\lambda \in R(G)$ denote the character of the irreducible representation of $G$ with highest weight $\lambda$. We define the \emph{multiplicity function} $m_G(k,\lambda)$ by the expression
\begin{gather}
\label{eqn:multfcn}
\index_G(D_{L^k})=\sum_{\lambda \in \Lambda \cap \t_+^\ast} m_G(k,\lambda)\chi_\lambda.
\end{gather}
An important theme in the work of Szenes--Vergne \cite{SzenesVergne2010} and also in our approach, is that the function $m_G(k,\lambda)$ has more coherent behavior than its restriction to any fixed value of $k$.

The statement of the result requires some further background on orbifolds, for which we refer the reader to, for example, \cite[Appendix A]{daSilvaThesis}, \cite[Section 2]{MeinrenkenSymplecticSurgery}. A small warning is that we will not require the action of isotropy groups in orbifold charts to be effective (this is in agreement with the references \cite{daSilvaThesis,MeinrenkenSymplecticSurgery} mentioned above). One advantage of permitting this, is that for a locally free action of a compact Lie group $K$ on a manifold $P$, the corresponding orbifold $P/K$ has orbifold charts given automatically by the slice theorem, with the isotropy groups being simply the isotropy groups for the action of $K$ on $P$.

In fact all the orbifolds that we will encounter arise naturally as such quotients $P/K$, and one could avoid mentioning orbifolds altogether by working instead with suitable $K$-basic structures on $P$. An example is the description of characteristic forms for orbifold vector bundles, which can be defined in terms of orbifold charts for~$P/K$, or alternatively in terms of $K$-basic differential forms on $P$. In brief, the latter approach goes as follows. One can take the complex $(\Omega_{\tn{bas}}(P),{\rm d})$ of $K$-basic differential forms on $P$ as a working definition of the de Rham complex of~$P/K$ (if~$K$ acts freely then $P/K$ is a manifold and pullback of forms from $P/K$ to $P$ is an isomorphism of complexes $(\Omega(P/K),{\rm d})\simeq (\Omega_{\tn{bas}}(P),{\rm d})$). A~$K$-equivariant vector bundle $E \rightarrow P$ determines an orbifold vector bundle $E/K$ over $P/K$. Let $\theta$ be a connection on $P$ with curvature $F_\theta$. The choice of connection determines a \emph{Cartan map} (cf.~\cite{MeinrenkenEncyclopedia}) from closed $K$-equivariant forms $\alpha(X)$ on $P$ to closed $K$-basic forms: $\alpha(X) \mapsto \tn{Car}_\theta(\alpha):=\Pi_{\tn{hor}}\alpha(F_\theta)$, where $\Pi_{\tn{hor}}$ is the projection onto the horizontal part relative to the connection. The Cartan map induces an isomorphism from the $K$-equivariant cohomology of $P$ to the cohomology of the complex of basic differential forms on~$P$. If $\alpha(X)$ is a~$K$-equivariant characteristic form (constructed via the $K$-equivariant analogue of the usual Chern--Weil construction cf.~\cite{BerlineGetzlerVergne,MeinrenkenEncyclopedia}), then one may take $\tn{Car}_\theta(\alpha) \in \Omega_{\tn{bas}}(P)$ as the definition of the corresponding characteristic form for $E/K$.

Let $\xi \in \Delta$ be a weakly-regular value. By the moment map equation \eqref{eqn:mm}, the action of $K=T/T_I$ on the level set
\[ P=\mu_\g^{-1}(\xi) \]
is locally free. The set $S_P$ of elements $g \in T/T_I$ such that $P^g\ne \varnothing$ is finite. For each $g \in S_P$, we obtain an orbifold
\[ \Sigma_g=P^g/(T/T_I), \qquad \Sigma=\bigsqcup_{g \in S_P} \Sigma_g.\]
Note that $\Sigma_1=P/(T/T_I)=M_\xi$ identifies with the reduced space itself, and more generally $\Sigma_g$ identifies with a symplectic quotient of $Y^g$. For each $g \in S_P$ there is an immersion $\Sigma_g\hookrightarrow \Sigma$ induced by $P^g\hookrightarrow P$. Let $\nu_{\Sigma_g,\Sigma}$ denote the (orbifold) normal bundle (the quotient $\nu_{P^g,P}/(T/T_I)$), which inherits a complex structure from the almost complex structures on $Y$, $Y^g$. Define the characteristic form
\[ \calDC^g(\nu_{\Sigma_g,\Sigma})=\tn{det}_\bC\big(1-g_\nu^{-1} {\rm e}^{-\frac{\i}{2\pi}F_\nu}\big), \]
where $g_\nu$ denotes the action of $g$ on the normal bundle (defined in terms of an orbifold chart, or in terms of $\nu_{P^g,P}$), and $F_\nu$ denotes the curvature. Taking the quotient of $L|_{P^g}$ we obtain (orbifold) line bundles
\[ L_{\Sigma_g}=(L|_{P^g})/(T/T_I), \qquad L_\Sigma=\bigsqcup_{g \in S_P} L_{\Sigma_g}. \]
There is a locally constant function
\[ g_L \colon \ \Sigma_g \rightarrow U(1) \]
giving the phase of the action of $g$ on $L_{\Sigma_g}$ (or equivalently on $L|_{P^g}$). Let $d \colon \Sigma \rightarrow \bZ$ be the locally constant function giving the size of a generic isotropy group for $\Sigma$ (or equivalently the number of elements in the generic stabilizer for the $T/T_I$ action on $\sqcup P^g$).

Let $\theta$ be a connection for the locally free $K=T/T_I$-action on $\sqcup_{g \in S_P}P^g$. The curvature $F_\theta$ is horizontal and $\t/\t_I$-valued, hence for any $\lambda \in (\t/\t_I)^\ast=I$, the form $\pair{\lambda}{F_\theta}$ is $K$-basic, hence descends to $\Sigma$. With the preparations above, we can state the main result of this note.
\begin{Theorem}\label{thm:main}If $0 \notin \Delta$ then $m(k,0)=0$ for all $k \ge 1$. If $0 \in \Delta$ then there is a closed polytope $\p \subset \Delta$ of the same dimension as $\Delta$ and containing the origin such that the following is true. Let $C_\p$ denote the cone
\[ C_\p=\{(t,t\tau)\,|\,t \in (0,\infty), \tau \in \p \} \subset \bR \times \t^\ast.\]
Fix a weakly regular value $\xi \in \Delta$ sufficiently close to $0$ as in Theorem~{\rm \ref{thm:MeinSja}}. Let $P=\mu_\g^{-1}(\xi)$ and define $\Sigma$, $L_\Sigma$, etc.\ as above. Then for all $(k,\lambda) \in (\bZ_{>0}\times \Lambda)\cap C_\p$,
\begin{gather}\label{eqn:main}
m_G(k,\lambda)=\sum_{g \in S_P} g^{-\lambda} \int_{\Sigma_g} \frac{1}{d}\frac{g_L^k\Ch(L_\Sigma)^k\Td(\Sigma)}{\calDC^g(\nu_{\Sigma_g,\Sigma})}{\rm e}^{\pair{\lambda}{F_\theta}}.
\end{gather}
\end{Theorem}
Of course this result is also originally due to Meinrenken--Sjamaar~\cite{MeinrenkenSjamaar}. Theorem~\ref{thm:MeinSja} follows immediately from Theorem~\ref{thm:main} by applying Kawasaki's index theorem for orbifolds to $\index\big(D^{\tn{red}}_{L_\xi}\big)$ and comparing with the evaluation of~\eqref{eqn:main} at $(k,\lambda)=(1,0)$.

Let us give a brief summary of our approach to deriving Theorem~\ref{thm:main}. Recall that a func\-tion~$f$ on a lattice $\Gamma$ in a real vector space~$V$ is said to be \emph{quasi-polynomial} if there is a~sublattice~$\Gamma^\prime$ with~$\Gamma/\Gamma^\prime$ finite and $f$ restricts to a polynomial function on each coset of~$\Gamma^\prime$. More generally, one says~$f$ is quasi-polynomial on a subset $\Gamma_0 \subset \Gamma$ if $f\upharpoonright \Gamma_0=q\upharpoonright \Gamma_0$ for some quasi-polynomial~$q$. A~fundamental fact, originally derived from Theorem~\ref{thm:MeinSja} by Meinrenken--Sjamaar~\cite{MeinrenkenSjamaar}, is that~$m_G$ is quasi-polynomial on the subset $C_\p \cap (\bZ_{>0}\times \Lambda)$. Our first goal, in Section~\ref{sec:mult}, is to give an independent proof of this fact, taking as a~starting point a~formula for~$m_G$ due to Szenes--Vergne~\cite{SzenesVergne2010} (inspired by work of Paradan~\cite{ParadanRiemannRoch}), which they obtained by a~combinatorial rearrangement of the fixed-point formula for the index.

Then in Section \ref{sec:statphase} we adapt an idea of Meinrenken~\cite{MeinrenkenAsymptotic} to compute the quasi-polynomial $m_G \upharpoonright C_\p$ using the Berline--Vergne index formula and the principle of stationary phase. The output of the stationary phase formula is an asymptotic expansion for $m_G(k,k\xi)$ in powers of~$k$ (allowing coefficients that are periodic in~$k$). As one knows in advance that $m_G(k,k\xi)$ is quasi-polynomial in~$k$, one concludes that the expansion is exact, yielding Theorem~\ref{thm:main}.

The article of Meinrenken--Sjamaar~\cite{MeinrenkenSjamaar} contains, besides Theorem~\ref{thm:MeinSja}, a wealth of detailed information about singular reduction and $[Q,R]=0$. Our goal in this short note is much more modest. We also do not make a great claim of originality, and in particular the debt to~\cite{SzenesVergne2010} and~\cite{MeinrenkenAsymptotic} will be apparent. Part of our motivation stems from the hope that the article of Szenes--Vergne~\cite{SzenesVergne2010}, in combination with this note, will provide a more elementary treatment of the $[Q,R]=0$ theorem than was previously available.

\section{Quasi-polynomials and the multiplicity function}\label{sec:mult}
The goal of this section is Theorem~\ref{thm:qp} on the quasi-polynomial behavior of the multiplicity function, which we prove using results of Szenes--Vergne~\cite{SzenesVergne2010} reviewed below.

The quotient $\g/\t$ can be identified with the unique $\Ad(T)$-invariant complement to $\t$ in $\g$. Let $\t \subset \h \subset \g$ be a $T$-invariant subspace. We may similarly identify $\h/\t$ and $\g/\h$ with subspaces of~$\g$. The choice of positive roots $\R_+$ determines a complex structure on $\g/\t$, whose $+\i$-eigenspace is identified with the direct sum of the positive root spaces:
\[ (\g/\t)^{1,0}\simeq \bigoplus_{\alpha \in \R_+} \g_\alpha.\]
We obtain similar complex structures on $\g/\h$, $\h/\t$, whose $+\i$-eigenspaces are direct sums of positive roots spaces. We will write $\de^{\g/\t}(a)$ (resp.~$\de^{\g/\h}(a)$, $\de^{\h/\t}(a)$) for the determinant of a complex linear endomorphism $a$ of $\g/\t$ (resp.~$\g/\h$,~$\h/\t$). An example is the endomorphism~$\Ad_t$, $t \in T$ (resp.~$\ad_X$, $X \in \t$); in this case we will simply write $\de^{\g/\t}(t)$ instead of $\de^{\g/\t}(\Ad_t)$ (resp.~$\de^{\g/\t}(X)$ instead of $\de^{\g/\t}(\ad_X)$), the action of $T$ (resp.~$\t$) on $\g/\t$ being understood. Then for example if $t={\rm e}^X \in T$,
\begin{gather*} \de^{\g/\t}\big(1-t^{-1}\big)=\prod_{\alpha \in \R_+}\big(1-t^{-\alpha}\big)=\prod_{\alpha \in \R_+}\big(1-{\rm e}^{-2\pi \i \pair{\alpha}{X}}\big), \\
 \de^{\g/\t}(-X)=\prod_{\alpha \in \R_+} -2\pi \i \pair{\alpha}{X}.\end{gather*}

For $\lambda \in \Lambda\cap \t_+^\ast$, the Weyl character formula says that for $t \in T$,
\begin{gather}\label{eqn:Weylchar}
\chi_\lambda(t) \cdot \de^{\g/\t}\big(1-t^{-1}\big)=\sum_{w\in W} (-1)^{l(w)}t^{w(\lambda+\rho)-\rho},
\end{gather}
where $W$ is the Weyl group, $l(w)$ is the length of the element $w \in W$, and $\rho$ is the half sum of the positive roots. The right-hand-side is an element of $R(T)$ with multiplicity function $m_\lambda$ obtained by Fourier transform. Note that
\begin{itemize}\itemsep=0pt
\item $m_\lambda$ is anti-symmetric under the $\rho$-\emph{shifted action of the Weyl group}:
\begin{gather*}
m_\lambda(w(\mu+\rho)-\rho)=(-1)^{l(w)}m_\lambda(\mu).
\end{gather*}
\item The support of $m_\lambda|_{\Lambda \cap \t_+^\ast}$ is $\{\lambda\}$, where it takes the value $1$.
\end{itemize}
Conversely these two properties determine $m_\lambda$: it is the unique $W$-anti-symmetric function on~$\Lambda$ extending the multiplicity function of $\chi_\lambda$. Applying these observations to the multiplicity function $m_G$ defined in~\eqref{eqn:multfcn}, we make the following definition.
\begin{Definition}Let $m(k,-)\colon \Lambda \rightarrow \bZ$ be the unique $\rho$-shifted $W$-anti-symmetric function such that $m(k,\lambda)=m_G(k,\lambda)$ for all $\lambda \in \Lambda \cap \t_+^\ast$. The corresponding character $Q(k,-) \colon T \rightarrow \bC$ is defined as the inverse Fourier transform:
\[ Q(k,t)=\sum_{\lambda \in \Lambda} m(k,\lambda)t^\lambda.\]
\end{Definition}

Using the Weyl character formula \eqref{eqn:Weylchar} and the definition of $m_G$, it is easy to verify that
\begin{gather*}
Q(k,t)=\sum_{\lambda \in \Lambda} m(k,\lambda)t^\lambda=\index_T(D_{L^k})(t)\cdot \de^{\g/\t}\big(1-t^{-1}\big).
\end{gather*}

We define
\[ \mu=\pr_{\t^\ast}\circ \mu_\g \]
to be the composition of the moment map $\mu_\g$ with the projection to $\t^\ast$. Then $\mu$ is a moment map for the action of~$T$ on~$M$. Suppose $t\in T$ is sufficiently generic, so that $M^t=M^T$. The Atiyah--Bott--Segal formula for the index yields
\begin{gather}\label{eqn:fixedpt}
Q(k,t)=\sum_{F \subset M^T}t^{k\mu_F}\int_F \frac{{\rm e}^{k\omega}\Td(F)}{\calDC^t(\nu_F)}\de^{\g/\t}\big(1-t^{-1}\big),
\end{gather}
where the sum is over connected components $F$ of $M^T$, and $\mu_F$ denotes the constant value of the moment map $\mu$ on $F$. The multiplicity $m$ is obtained by Fourier transform of~\eqref{eqn:fixedpt}.

Key to the approach in~\cite{SzenesVergne2010} is a different expression for $m(k,\lambda)$ that we briefly describe here. The formula depends on the choice of an invariant inner product on $\g$, as well as a generic point~$\gamma$ contained in $\t^\ast_+$ and sufficiently close to~$0$ (see~\cite[Section~4.1]{SzenesVergne2010} for the meaning of `generic' here). Using the inner product we identify $\t \simeq \t^\ast$. We need some additional notation:
\begin{itemize}\itemsep=0pt
\item Let $\tn{Comp}_T(M)$ denote the set of connected components of $M^H$, as $H$ ranges over all (connected) sub-tori of~$T$.
\item For $C \in \tn{Comp}_T(M)$, let $\t_C \subset \t$ be its generic infinitesimal stabilizer. Let $A_C$ be the smallest affine subspace containing the image $\mu(C)$. In particular $A_M$ is the smallest affine subspace containing $\mu(M)$. Note that $A_C$ is a translate of the annihilator of~$\t_C$.
\item Let $\gamma_C \in A_C$ be the orthogonal projection of $\gamma$ onto $A_C$, and let $\tau_C=\gamma_C-\gamma$.
\end{itemize}

The Szenes--Vergne--Paradan formula \cite[equation~(39)]{SzenesVergne2010} (see also \cite[Proposition~41, Theorem~48]{SzenesVergne2010}) is a sum of contributions:
\begin{gather}\label{eqn:SzVer}
m=\sum_C m_C,
\end{gather}
where $C$ ranges over components $C \in \tn{Comp}_T(M)$ such that $\gamma_C \in \mu_\g(C)$. Szenes--Vergne derive this formula directly from \eqref{eqn:fixedpt} using an interesting combinatorial rearrangement, the main ingredient of which is a decomposition formula for Kostant-type partition functions. The formula is inspired by, and closely related to, the work of Paradan~\cite{ParadanRiemannRoch}. The fact that only a subset of the components in $\tn{Comp}_T(M)$ contribute is non-trivial and quite important for $[Q,R]=0$. The proof given by Szenes--Vergne involves studying the asymptotic behavior of the~$m_C$'s using the Berline--Vergne formula and the principle of stationary phase. It goes back to results of Paradan~\cite{ParadanRiemannRoch}, who proved a closely related result using transversally elliptic symbols and K-theoretic methods. Note that Szenes--Vergne assume for simplicity that $M^T$ consists of isolated fixed points, but it is not difficult to handle the general case with the same methods; see for example \cite[Section~7]{VerlindeSums} for some indications of how this can be done.

For the proof of Theorem \ref{thm:qp} we do not need the precise definition of the terms $m_C$ in \eqref{eqn:SzVer}, but we will need the following two crucial properties:
\begin{enumerate}\itemsep=0pt
\item The function $m_C$ restricts to a quasi-polynomial on each $\Lambda$-translate of the set $(\bZ\times \Lambda)\cap \bm{A}_C$, where
\[ \bm{A}_C=\{(t,t\tau)\,|\,t\in \bR_{>0}, \,\tau \in A_C\} \subset \bR \times \t^\ast.\]
\item Let $\tn{wt}(\nu_C)$ denote the list of complex weights (for the compatible almost complex structure~$J$) for the $\t_C$ action on the normal bundle~$\nu_C$. If $\lambda \in \Lambda$ is in the support of~$m_C(k,-)$ then~$\lambda$ satisfies the inequality
\begin{gather}\label{eqn:d}
\pair{\tau_C}{\lambda}\ge k\pair{\tau_C}{\gamma_C}+\pair{\tau_C}{\sigma_C}, \qquad \sigma_C:=\sum_{\substack{\delta \in \tn{wt}(\nu_C)\\ \pair{\tau_C}{\delta}>0}} \delta-\sum_{\substack{\alpha \in \R_+ \\ \pair{\tau_C}{\alpha}>0}}\alpha.
\end{gather}
See the proof of \cite[Theorem~49]{SzenesVergne2010}. Note that, except for the special case $\tau_C=0$, \eqref{eqn:d}~defines a half-space in $\t^\ast$.
\end{enumerate}
We will refer to these two properties as `property~(a)', `property~(b)' in the proof of the next result. Theorem~\ref{thm:qp} is a strengthening of \cite[Theorem~49]{SzenesVergne2010} (which says that the function $k \mapsto m(k,0)$ is quasi-polynomial), and our arguments are based on their elegant approach.

\begin{Theorem}[\cite{MeinrenkenSjamaar}, see also \cite{ParadanRiemannRoch, ParadanWallCrossing, WittenNonAbelian}]\label{thm:qp}
If $0 \notin \Delta$ then $m(k,0)=0$ for all $k \ge 1$. If $0 \in \Delta$ then there is a closed polytope $\p \subset \Delta$ of the same dimension as $\Delta$ and containing the origin such that $m(k,\lambda)$ is quasi-polynomial on the set of integral points $(\bZ \times \Lambda) \cap C_\p$ contained in the cone
\[ C_\p=\{(t,t\tau)\,|\,t \in (0,\infty), \,\tau \in \p \} \subset \bR \times \t^\ast.\]
\end{Theorem}
\begin{proof}The strategy is based on choosing a suitable $\gamma \in \t^\ast_+$ and then analyzing the supports of the contributions $m_C$ to $m$ in the corresponding Szenes--Vergne--Paradan formula~\eqref{eqn:SzVer} using property~(b). The contribution $m_C$ appears in~\eqref{eqn:SzVer} only if $\gamma_C \in \mu_\g(M)\cap \t^\ast=W \cdot \Delta \subset W \cdot I$ (recall by definition~$I$ is the smallest affine subspace containing~$\Delta$). Because $\gamma$ is chosen generically, the only $C \in \tn{Comp}_T(M)$ which may contribute to~\eqref{eqn:SzVer} are those such that the affine subspace~$A_C$ is entirely contained in~$I$ or one of its Weyl reflections, and throughout the proof we assume this is the case.

Suppose $0 \in \Delta$. We argue that by a suitable choice of $\gamma$, one can arrange that \emph{for all but one} of the contributions, (i) $\pair{\tau_C}{\gamma_C}\ge 0$ with equality if and only if $0 \in A_C$, (ii) $\pair{\tau_C}{\gamma_C}>\pair{\tau_C}{\gamma_I}$, where $\gamma_I$ is the orthogonal projection of $\gamma$ onto $I$, and (iii) $\pair{\tau_C}{\sigma_C}>0$. The one special contribution is denoted $m_{C_I}$ below and corresponds to the subspace $A_{C_I}=I$. By property~(b), (i) and (iii) imply that for $C \ne C_I$, the support of $m_C(k,-)$ lies outside $kH_C$ where $H_C$ is the half-space
\[ H_C=\{\xi \,|\, \pair{\tau_C}{\xi}\le \pair{\tau_C}{\gamma_C}\}.\]
Let $\p$ be the intersection of $I$ with all of the half-spaces $H_C$ for $C \ne C_I$. By~(ii), the relative interior of $\p$, viewed as a polytope in $I$, contains the point $\gamma_I$, hence in particular is non-empty. By construction $m\upharpoonright C_\p=m_{C_I}\upharpoonright C_\p$. Then property~(a) implies that $m_{C_I}$ is quasi-polynomial on $C_\p$, hence the result.

We claim that one can ensure (i) holds for all $C$ by choosing $\gamma \in \t^\ast_+$ sufficiently close to $0$. Indeed let $A_C^0$ be the subspace parallel to $A_C$, and let $a_C \in A_C$ be the nearest point in $A_C$ to $0$. Then $\gamma_C-a_C \in A_C^0$ while $\tau_C$, $a_C$ are both orthogonal to $A_C^0$, hence $\pair{\tau_C}{\gamma_C-a_C}=0=\pair{a_C}{\gamma_C-a_C}$. These imply $\pair{\tau_C}{\gamma_C}=\|a_C\|^2-\pair{a_C}{\gamma}$. If $0 \in A_C$ then $a_C=0$ and this vanishes. Otherwise we can ensure $\pair{\tau_C}{\gamma_C}>0$ by choosing $\|\gamma\|<\|a_C\|$. Since only finitely many $C$ occur, we can choose $\gamma$ such that this holds for all $C$ with $0 \notin A_C$. We now turn to verifying (ii), (iii), and also handle the case $0 \notin \Delta$ along the way.

Suppose $\gamma_C \in \mu_\g(C)$, so that $m_C$ indeed appears in \eqref{eqn:SzVer}. If $\alpha \in \R_+$ and $\pair{\tau_C}{\alpha}>0$, then since $\gamma \in \t^\ast_+$ it follows that $\pair{\gamma_C}{\alpha}>0$. It is a consequence of the cross-section theorem (cf.~\cite{ConvexitySympCuts}) that $\alpha|_{\t_C}$ appears in the list of weights $\tn{wt}(\nu_C)$. Hence
\begin{gather}\label{eqn:d2}
\sigma_C=\sum_{\substack{\delta \in \tn{wt}(\nu_C)-\R_+^{\tau_C}\\ \pair{\tau_C}{\delta}>0}} \delta,
\end{gather}
where $\R_+^{\tau_C}$ denotes the set of positive roots $\alpha$ such that $\pair{\tau_C}{\alpha}>0$, and $\tn{wt}(\nu_C)-\R_+^{\tau_C}$ denotes the list of weights on $\nu_C$ with one copy of $\alpha|_{\t_C}$ removed for each $\alpha \in \R_+$ satisfying $\pair{\tau_C}{\alpha}>0$. Hence{\samepage
\begin{gather}\label{eqn:tausignonneg}
\pair{\tau_C}{\sigma_C}\ge 0
\end{gather}
and the inequality is strict if at least one weight $\delta$ contributes in \eqref{eqn:d2}.}

If $0 \notin \Delta$ then, choosing $\gamma$ sufficiently close to $0$, we can ensure that for each $C$ such that $0 \in A_C$ we have $\gamma_C \notin \mu_\g(M)$ (a fortiori $\gamma_C \notin \mu_\g(C)$), hence $m_C$ does not appear in \eqref{eqn:SzVer} at all. On the other hand, by~(i), \eqref{eqn:tausignonneg} and property~(b), if $0 \notin A_C$ then $m_C(k,0)=0$ for all $k \ge 1$. We conclude that if $0 \notin \Delta$ then $m(k,0)=0$ for all $k \ge 1$.

We turn to the case $0 \in \Delta \subset I$. In this case we may choose $\gamma$ such that it is simultaneously close to $0$ and arbitrarily close to $\gamma_I$, the orthogonal projection of $\gamma$ onto $I$. Since $\tau_C=\gamma_C-\gamma$, $\pair{\tau_C}{\gamma}\le \pair{\tau_C}{\gamma_C}$ with equality if and only if $\gamma_C=\gamma$. By taking $\gamma$ sufficiently close to $I$, one can ensure that $\pair{\tau_C}{\gamma_I}\le \pair{\tau_C}{\gamma_C}$ with equality if and only if $\gamma_C=\gamma_I$.

We first consider contributions from components $C \in \tn{Comp}_T(M)$ such that $\gamma_C \notin \t_+^\ast$. In this case there exists a negative root $\alpha \in \R_-$ such that $\pair{\gamma_C}{\alpha}>0$. It follows from the cross-section theorem that $\alpha|_{\t_C} \in \tn{wt}(\nu_C)$. Since $\gamma \in \t_+^\ast$, $\pair{\gamma}{\alpha}\le 0$ and so
\[ \pair{\tau_C}{\alpha}=\pair{\gamma_C}{\alpha}-\pair{\gamma}{\alpha}>0. \]
As $\alpha \notin \R_+$, we see that $\delta=\alpha$ indeed contributes in \eqref{eqn:d2}, hence $\pair{\tau_C}{\sigma_C}>0$. Moreover since $\gamma_C \notin \t_+^\ast$, $\gamma_C \ne \gamma_I$, hence $\pair{\tau_C}{\gamma_I}<\pair{\tau_C}{\gamma_C}$. This establishes (ii), (iii) for this case.

We are left to consider contributions from $C \in \tn{Comp}_T(M)$ such that $\gamma_C \in \Delta=\mu_\g(M)\cap \t_+^\ast$. Let $\Delta_{\tn{reg}} \subset \Delta$ be the relatively open dense subset of weakly regular values. The connected components of $\Delta_{\tn{reg}}$ are relatively open polytopes inside the subspace $I$. Choose a connected component $\a \subset \Delta_{\tn{reg}}$ containing $0$ in its closure. We may choose $\gamma \in \t^\ast_+$ such that the orthogonal projection $\gamma_I$ onto $I$ lies in $\a$. The fibre $\mu_\g^{-1}(\gamma_I)$ is connected and contained in $M^{T_I}$, hence there is a unique connected component $C_I \subset M^{T_I}$ containing $\mu_\g^{-1}(\gamma_I)$. Then $A_{C_I}=I$ and by property~(a), $m_{C_I}$ is quasi-polynomial on the set of integral points in $\bm{A}_C=\{(t,t\tau)\,|\,t>0,\,\tau \in I\} \supset C_\p$.

The final situation to consider consists of the contributions from $C \in \tn{Comp}_T(M)$ such that $\gamma_C \in \Delta \setminus \Delta_{\tn{reg}}$. In particular $\gamma_C \ne \gamma_I$ hence
\begin{gather}\label{eqn:iistrict}
\pair{\tau_C}{\gamma_I}<\pair{\tau_C}{\gamma_C}
\end{gather}
establishing (ii) for this case. Let $\sigma$ be the face of $\t^\ast_+$ containing $\gamma_C$. The subset
\[ U=G_\sigma\cdot \bigcup_{\ol{\tau}\supset \sigma} \tau, \]
where the union is taken over relatively open faces of $\t^\ast_+$ whose closure contains~$\sigma$, is a slice for the coadjoint $G_\sigma$-action. Let $Y=\mu_\g^{-1}(U)$ be the corresponding symplectic cross-section, cf.\ \cite[Remark~3.7, Theorem~3.8]{ConvexitySympCuts}. Consider the function $f=\pair{\tau_C}{\mu}|_Y \colon Y \rightarrow \bR$, for which $C \cap Y \subset Y^{\tau_C}=\tn{Crit}(f)$ is a critical submanifold. Note that $f|_{C\cap Y}=\pair{\tau_C}{\gamma_C}$. A result from symplectic geometry says that in a suitable tubular neighborhood of $C \cap Y$, the function $f$ takes the form
\begin{gather}\label{eqn:fcnf}
f(z_1,\dots,z_n)=\pair{\tau_C}{\gamma_C}-\pi \sum_j |z_j|^2\pair{\tau_C}{\delta_j},
\end{gather}
where $\delta_j \in \tn{wt}(\nu_{C\cap Y,Y})$, $\pair{\tau_C}{\delta_j}\ne 0$, $z_j$ is a vector in the subbundle of $\nu_{C\cap Y,Y}$ where $\t_C$ acts with weight $\delta_j$, and $|z_j|$ denotes its norm with respect to a suitable Hermitian structure.

Let $S$ be the line segment with endpoints $\gamma_I$ and $\gamma_C$. By convexity $S \subset \Delta$. The inverse image $\mu_\g^{-1}(S) \subset Y$ is connected since $\mu_\g$ has connected fibres. By~\eqref{eqn:iistrict}, along the line segment~$S$, $f$~varies between its absolute minimum $\pair{\tau_C}{\gamma_I}$ on the fibre $\mu_\g^{-1}(\gamma_I)$ and its absolute maximum $\pair{\tau_C}{\gamma_C}$ on the fibre $\mu_\g^{-1}(\gamma_C)$. By connectedness of $\mu_\g^{-1}(S)$ and equation~\eqref{eqn:fcnf}, there must exist a~$\delta_j$ such that $\pair{\tau_C}{\delta_j}>0$.

By the cross-section theorem $\nu_{Y,M}|_{C\cap Y}\simeq (C\cap Y)\times \g_{\gamma_C}^\perp$, where the orthogonal comple\-ment~$\g_{\gamma_C}^\perp$ is embedded in $TM|_{C\cap Y}$ as the orbit directions. The weights $\R_+^{\tau_C}$ which are removed in~\eqref{eqn:d2} can be identified with the weights of the $\t_C$-action on $\nu_{Y,M}|_{C\cap Y}$. With this understanding we have $\tn{wt}(\nu_{C\cap Y,Y})\subset \tn{wt}(\nu_C)-\R_+^{\tau_C}$. Thus $\delta_j$ indeed contributes to~\eqref{eqn:d2}, establishing~(iii) for this case. This completes the proof.
\end{proof}

\begin{Corollary}\label{cor:asymptoticdetermination}Suppose $0 \in \Delta$ and let $\p \subset \Delta$ be as in Theorem~{\rm \ref{thm:qp}}. If $\xi \in \p$ is rational and $n_\xi \in \bZ_{>0}$ is the least positive integer such that $n_\xi \xi \in \Lambda$, then the function
\[ f_\xi \colon \ n_\xi\cdot \bZ_{>0}\rightarrow \bZ, \qquad f_\xi(k)=m(k,k\xi) \]
is quasi-polynomial. Moreover $m\upharpoonright C_\p$ is the unique quasi-polynomial function such that $m(k,k\xi)\allowbreak =f_\xi(k)$ for all rational, weakly regular values $\xi$ in the relative interior of $\p$.
\end{Corollary}
\begin{Remark}A suitable finite collection of the functions $f_\xi$ already fully determines $m \upharpoonright C_\p$.
\end{Remark}

\section{Stationary phase calculation}\label{sec:statphase}
Assume $0 \in \Delta$ and let $\p \subset \Delta$ be as in Theorem \ref{thm:qp}, so that $m \upharpoonright C_\p$ is quasi-polynomial. By Corollary \ref{cor:asymptoticdetermination}, $m \upharpoonright C_\p$ is completely determined by the collection of quasi-polynomial functions $f_\xi(k)=m(k,k\xi)$, for $\xi$ ranging over rational, weakly regular values of $\mu_\g$ lying in the relative interior of $\p$. In this section we use the Berline--Vergne index formula and the stationary phase expansion to compute the functions $f_\xi$, and hence also $m \upharpoonright C_\p$. The end result will be the formula \eqref{eqn:main} in Theorem \ref{thm:main}.

Let $t \in T$. By the Berline--Vergne formula, for $X \in \t$ sufficiently small one has $Q\big(k,t{\rm e}^X\big)=Q_t(k,X)$ where
\begin{gather}\label{eqn:BerVer}
Q_t(k,X):=\int_{M^t} \frac{t_L^k {\rm e}^{k(\omega+2\pi \i\pair{\mu}{X})}\Td\big(M^t,\tfrac{2\pi}{\i}X\big)}{\calDC^t\big(\nu_{M^t,M},\tfrac{2\pi}{\i}X\big)}\de^{\g/\t}\big(1-t^{-1}{\rm e}^{-X}\big),
\end{gather}
and $\Td\big(M^t,\tfrac{2\pi}{\i}X\big)$, $\calDC^t\big(\nu_{M^t,M},\tfrac{2\pi}{\i}X\big)$ denote equivariant extensions of the usual Chern-Weil forms, closed with respect to the differential ${\rm d}+2\pi \i \iota(X_M)$, obtained by replacing curvatures with equivariant curvatures (evaluated at $\tfrac{2\pi}{\i}X$) in the usual formulas (cf.~\cite{BerlineGetzlerVergne} for details, although note that we are using the topologist's convention for characteristic classes).

Let $B_r$ denote the ball of radius $r>0$ around the origin in $\g/\t$. Let $\mu_{\g/\t}$ denote the composition of $\mu_\g$ with the quotient map $\g \rightarrow \g/\t$. Let $\g^t \subset \g$ denote the fixed-point set of $\Ad_t$. Then~$B_r^t$ is a neighborhood of $0$ in $\g^t/\t$. Recall $\g^t/\t$, $\g/\g^t$ are equipped with complex structures such that their $+\i$-eigenspaces are identified with sums of positive root spaces. Equip $\g^t/\t$ with the orientation induced by the complex structure, and let $\tau_{\g^t/\t}(X)$ be a $T$-equivariant Thom form with support contained in $B_r^t$, closed for the differential ${\rm d}-\iota(X_M)$. Consider the $T$-equivariant differential form on $\g^t/\t$ (closed for the differential ${\rm d}+2\pi \i \iota(X_{\g^t/\t})$) given by
\[ \Ch^t\big({\sf b},\tfrac{2\pi}{\i}X\big)=\de^{\g/\g^t}\big(1-t^{-1}{\rm e}^{-X}\big)\de^{\g^t/\t}\left(\frac{1-{\rm e}^{-X}}{X}\right)\tau_{\g^t/\t}\big(\tfrac{2\pi}{\i}X\big),\]
The map $\mu_{\g/\t}$ restricts to a map $M^t \rightarrow \g^t/\t$, which we use to pull back the form $\Ch^t\big({\sf b},\tfrac{2\pi}{\i}X\big)$.
\begin{Lemma}
\begin{gather}\label{eqn:Qg}
Q_t(k,X)=\int_{\mu_{\g/\t}^{-1}(B_r^t)} \frac{t_L^k {\rm e}^{k(\omega+2\pi \i\pair{\mu}{X})}\Td\big(M^t,\tfrac{2\pi}{\i}X\big)}{\calDC^t\big(\nu_{M^t,M},\tfrac{2\pi}{\i}X\big)}\Ch^t\big({\sf b},\tfrac{2\pi}{\i}X\big).
\end{gather}
\end{Lemma}
\begin{proof}The pullback of $\tau_{\g^t/\t}(X)$ to $0 \in \g^t/\t$ is the equivariant Euler class, which (since $0$ is just a point) is the function
\[ \prod_{\alpha \in \R^{\g^t}_+} -\pair{\alpha}{X}=\de^{\g^t/\t}\big(\tfrac{\i}{2\pi}X\big),\]
where $\R^{\g^t}_+\subset \R_+$ is a set of positive roots for $\g^t$. Note also that $t$ acts trivially on $\g^t/\t$, since $\g^t$ is the fixed point subspace under the adjoint action. It follows that the pullback of $\Ch^t\big({\sf b},\tfrac{2\pi}{\i}X\big)$ to $0 \in \g^t/\t$ is the function $\de^{\g/\t}\big(1-t^{-1}{\rm e}^{-X}\big)$. Since pullback to $\{0\}=(\g^t/\t)^T$ is injective on equivariant cohomology classes, $\Ch^t\big({\sf b},\tfrac{2\pi}{\i}X\big)$, $\de^{\g/\t}\big(1-t^{-1}{\rm e}^{-X}\big)$ determine the same class in $T$-equivariant cohomology of~$\g^t/\t$. As~$M$ is compact, we may make this replacement in~\eqref{eqn:BerVer} without changing the value of the integral.
\end{proof}
\begin{Remark}
The reason for the notation is that $\Ch^t\big({\sf b},\tfrac{2\pi}{\i}X\big)$ is a representative for the $t$-twisted Chern character of a Bott element ${\sf b} \in K_T^0(\g/\t)$, which generates the latter as an $R(T)=K_T^0(\pt)$-module. To be more precise, ${\sf b}$ is the generator whose pullback to $0 \in \g/\t$ is $[\wedge^{\tn{ev}} \n_-]-[\wedge^{\tn{odd}}\n_-] \in K^0_T(\pt)$, $\n_-$ being the direct sum of the negative root spaces.
\end{Remark}

Since $T$ is compact, there exists a finite set $S \subset T$ and an open cover $\{U_t\,|\,t \in S\}$ of~$T$ where~$U_t$ is a small open ball around~$t$ in $T$ such that $Q\big(k,t{\rm e}^X\big)=Q_t(k,X)$ for $t{\rm e}^X \in U_t$. Let~$\sigma_t$, $t \in S$ be bump functions on $\t$ such that $\{\hat{t}_\ast \sigma_t\,|\,t \in S\}$ is a partition of unity subordinate to the cover, where $\hat{t}$ is the map
\[ \hat{t}\colon \ \t \rightarrow T, \qquad X \mapsto t{\rm e}^X,\]
which we may assume restricts to a diffeomorphism of a small ball around $0 \in \t$ onto $U_t$. By equations \eqref{eqn:BerVer} and \eqref{eqn:Qg}
\[ Q=\sum_{t \in S} \hat{t}_\ast (\sigma_t Q_t). \]

The multiplicity function $m$ is the Fourier transform of $Q$:
\[ m(k,\lambda)=\sum_{t\in S} \int_{\t} \sigma_t(X)\big(t{\rm e}^X\big)^{-\lambda}Q_t(k,X).\]
To do the stationary phase calculation (for $k \rightarrow \infty$) following the approach outlined at the beginning of this section, we now set $\lambda=k\xi$ where $\xi \in (\Lambda \otimes \bQ)\cap \p$ is a rational, weakly regular value of $\mu_\g$ contained in the relative interior of $\p \subset \Delta$ as in Corollary~\ref{cor:asymptoticdetermination}, $k \in n_\xi\bZ_{>0}$ and $n_\xi$ is the least positive integer such that $n_\xi\xi \in \Lambda$. Thus
\begin{gather}\label{eqn:0}
m(k,k\xi)=\sum_t t^{-k\xi}\!\!\int_\t {\rm d} X \sigma_t(X)\!\! \int_{\mu_{\g/\t}^{-1}(B_r^t)} \frac{t_L^k\Td\big(M^t,\tfrac{2\pi}{\i}X\big)}{\calDC^t\big(\nu_{M^t,M},\tfrac{2\pi}{\i}X\big)}\Ch^t\big({\sf b},\tfrac{2\pi}{\i}X\big){\rm e}^{k(\omega+2\pi \i\pair{\mu-\xi}{X})}.\!\!\!\!
\end{gather}

Let $f(m,X)=\pair{\mu(m)-\xi}{X}$ viewed as a real-valued function on $\mu_{\g/\t}^{-1}\big(B_r^t\big) \times \t$. According to the principle of stationary phase, we can include a bump function supported in a small neighborhood of the critical set of $f$ in the integrand of \eqref{eqn:0}, and the error will be $o(k^{-\infty})$. The derivative
\[ {\rm d}_{(m,X_0)}f=\pair{{\rm d}_m\mu}{X_0}+\pair{\mu(m)-\xi}{{\rm d}_{X_0}X}\]
and in particular $\tn{Crit}(f)\subset \mu^{-1}(\xi)\times \t$. Let $\chi$ be the pullback by $\mu$ of a bump function in $\t^\ast$ supported in a small neighborhood of $\xi$. Thus{\samepage
\begin{gather}
m(k,k\xi)\sim \sum_t t^{-k\xi} \int_\t {\rm d}X \sigma_t(X)\nonumber\\
\hphantom{m(k,k\xi)\sim}{}\times \int_{\mu_{\g/\t}^{-1}(B_r^t)}\chi \frac{t_L^k \Td\big(M^t,\tfrac{2\pi}{\i}X\big)}{\calDC^t\big(\nu_{M^t,M},\tfrac{2\pi}{\i}X\big)}\Ch^t\big({\sf b},\tfrac{2\pi}{\i}X\big){\rm e}^{k(\omega+2\pi \i\pair{\mu-\xi}{X})},\label{eqn:1}
\end{gather}
where $\sim$ denotes equality modulo an $o(k^{-\infty})$ error.}

Let $Y=\mu_{\g}^{-1}(\sigma)$ be the cross-section for the principal face. By the cross-section theorem, a~neighborhood $N$ of $Y$ in $M$ is $G_\sigma$-equivariantly diffeomorphic to
\[ Y\times \g/\g_\sigma, \]
where $\g/\g_\sigma \simeq \g_\sigma^\perp$ is embedded in the orbit directions. Since $\mu^{-1}(\xi)\cap \mu_\g^{-1}(\t^\ast)=\mu_\g^{-1}(\xi)\subset Y$, by taking $r$ and $\supp(\chi)$ sufficiently small, we can assume that $\supp(\chi)\cap \mu_{\g/\t}^{-1}(B_r)$ is contained in a small neighborhood of $\mu_\g^{-1}(\xi)$ where the local model $Y\times \g/\g_\sigma$ is valid, and so we may replace $\mu_{\g/\t}^{-1}\big(B_r^t\big)$ with $N^t$ in equation~\eqref{eqn:1}. In the next lemma we use the Thom form to integrate over the $(\g/\g_\sigma)^t$ directions.
\begin{Lemma}
\begin{gather}
m(k,k\xi)\sim \sum_{t \in S}t^{-k\xi} \int_\t {\rm d}X \sigma_t(X)\nonumber\\
\hphantom{m(k,k\xi)\sim}{}\times \int_{Y^t} \chi \frac{t_L^k\Td\big(Y^t,\tfrac{2\pi}{\i}X\big)}{\calDC^t\big(\nu_{Y^t,Y},\tfrac{2\pi}{\i}X\big)}{\rm e}^{k(\omega+2\pi \i\pair{\mu-\xi}{X})}\de^{\g_\sigma/\t}\big(1-t^{-1}{\rm e}^{-X}\big).\label{eqn:2}
\end{gather}
\end{Lemma}
\begin{proof}The neighborhood $N^t$ of $Y^t$ in $M^t$ is $T$-equivariantly diffeomorphic to
\[ Y^t \times (\g/\g_\sigma)^t=Y^t\times \g^t/\g_\sigma^t,\]
where $\g_\sigma^t=(\g_\sigma)^t$ is the subspace of $\g_\sigma$ fixed by $t$. Moreover the almost complex structure on $N^t$ is homotopic to a product almost complex structure, where $Y^t$ is equipped with an almost complex structure compatible with the symplectic form in the cross-section, and $\g^t/\g_\sigma^t$ is equipped with the almost complex structure whose $+\i$-eigenspace is identified with a sum of positive root spaces. Let
\[ p\colon \ N^t \rightarrow Y^t \]
denote the projection. For the normal bundle
\[ \nu_{M^t,M}|_{N^t} \simeq p^\ast \nu_{Y^t,Y}\oplus (\g/\g_\sigma)/(\g/\g_\sigma)^t=p^\ast \nu_{Y^t,Y}\oplus\g/\big(\g^t+\g_\sigma\big),\]
and again the almost complex structure is homotopic to a product one, using a compatible almost complex structure on the symplectic vector bundle $\nu_{Y^t,Y}$, and an almost complex structure on~$\g/\big(\g^t+\g_\sigma\big)$ whose $+\i$-eigenspace is identified with a sum of positive root spaces. Using the identifications above we obtain, up to equivariantly exact forms:
\begin{gather}\label{eqn:dec}
\Td\big(M^t,\tfrac{2\pi}{\i}X\big)|_{N^t} =\Td\big(Y^t,\tfrac{2\pi}{\i}X\big)\tn{det}_{\bC}^{\g^t/\g_\sigma^t}\left(\frac{X}{1-{\rm e}^{-X}}\right),\\ \nonumber
\calDC^t\big(\nu_{M^t,M},\tfrac{2\pi}{\i}X\big)|_{N^t} =\calDC^t\big(\nu_{Y,Y^t},\tfrac{2\pi}{\i}X\big)\tn{det}_{\bC}^{\g/(\g^t+\g_\sigma)}\big(1-t^{-1}{\rm e}^{-X}\big).
\end{gather}

Since $Y^t \subset \mu_\g^{-1}(\t^\ast)$, the pullback of the equivariant Thom form $\tau_{\g^t/\t}(X)$ to $Y^t$ is just the function
\[ \de^{\g^t/\t}\big(\tfrac{\i}{2\pi}X\big)=\de^{\g^t/\g_\sigma^t}\big(\tfrac{\i}{2\pi}X\big)\de^{\g_\sigma^t/\t}\big(\tfrac{\i}{2\pi}X\big).\]
We recognize $\de^{\g^t/\g_\sigma^t}\big(\tfrac{\i}{2\pi}X\big)$ as the equivariant Euler class of the trivial bundle $Y^t \times \g^t/\g_\sigma^t$. Thus up to an equivariantly exact form, we have{\samepage
\begin{gather}\label{eqn:dec2}
\tau_{\g^t/\t}(X)=\tau_p(X)\de^{\g_\sigma^t/\t}\big(\tfrac{\i}{2\pi}X\big),
\end{gather}
where $\tau_p(X)$ is an equivariant Thom form for the vector bundle $p\colon N^t=Y^t\times \g^t/\g_\sigma^t\rightarrow Y^t$.}

We next want to make the replacements \eqref{eqn:dec}, \eqref{eqn:dec2} in equation \eqref{eqn:1}, and then use the Thom form $\tau_p(X)$ to integrate over the fibres of $p\colon N^t \rightarrow Y^t$. In the integral over $N^t$ in \eqref{eqn:1}, the integrand has compact support and all terms in the integrand are equivariantly closed except for the bump function $\chi$. By Stokes' theorem, replacing a form by a cohomologous form in the integrand leads to an error term containing ${\rm d}\chi$; but ${\rm d}\chi$ vanishes near $\mu^{-1}(\xi)$, so the principle of stationary phase implies the error will be $o(k^{-\infty})$. Let $\iota_{Y^t}\colon Y^t\hookrightarrow N^t$ denote the inclusion. Similarly the formula $p_\ast(\tau_p(X)\alpha(X))=\iota_{Y^t}^\ast \alpha(X)$ applies when $\alpha(X)$ is equivariantly closed. But writing $\chi=1-(1-\chi)$, the principle of stationary phase again shows that we can make this replacement up to an $o(k^{-\infty})$ error term.

After making these replacements and integrating over the fibre, the form $\tau_p\big(\tfrac{2\pi}{\i}X\big)$ disappears. There are various Lie theoretic factors left over:
\[ \frac{\de^{\g/\g^t}\big(1-t^{-1}{\rm e}^{-X}\big)}{\de^{\g/(\g^t+\g_\sigma)}\big(1-t^{-1}{\rm e}^{-X}\big)}\de^{\g^t/\t}\left(\frac{1-{\rm e}^{-X}}{X}\right)\de^{\g_\sigma^t/\t}(X)\de^{\g^t/\g_\sigma^t}\left(\frac{X}{1-{\rm e}^{-X}}\right),\]
which simplify to $\de^{\g_\sigma/\t}\big(1-t^{-1}{\rm e}^{-X}\big)$ (one uses that $t$ acts trivially on $\g_\sigma^t/\t$ and that $\big(\g^t+\g_\sigma\big)/\g^t\simeq \g_\sigma/\g_\sigma^t$).
\end{proof}

Choose a complementary subtorus $T_I^\prime$ so that $T\simeq T_I \times T_I^\prime$. The quotient map $T \rightarrow T/T_I$ induces an isomorphism of groups $T_I^\prime \xrightarrow{\sim}T/T_I$.
By adding additional points if necessary, we may assume the finite subset $S\subset T$ is a product $S_I\times S_I^\prime$, where $S_I \subset T_I$, $S_I^\prime \subset T_I^\prime$ and that the image of $S_I^\prime$ in $T/T_I$ contains the set $S_P$ from the introduction. Thus we will write elements of~$S$ as products $hg$ with $h \in S_I \subset T_I$ and $g \in S_I^\prime \subset T_I^\prime$. We may assume the bump function $\sigma_t$ is a product $\sigma_h \cdot \sigma_g$, where $\sigma_h$ (resp.~$\sigma_g$) is a bump function on $\t_I$ (resp.~$\t_I^\prime$), satisfying
\begin{gather}\label{eqn:Sum1}
 \sum_{h \in S_I} \hat{h}_\ast\sigma_h=1, \qquad \sum_{g \in S_I^\prime} \hat{g}_\ast\sigma_g=1.
\end{gather}
The next lemma gives a further simplification of \eqref{eqn:2}.
\begin{Lemma}
\begin{gather}
\label{eqn:Asympt2}
m(k,k\xi)\sim \sum_{g \in S_I^\prime}g^{-k\xi} \int_{\t_I^\prime}{\rm d} X\sigma_g(X) \int_{Y^g} \chi \frac{g_L^k\Td\big(Y^g,\tfrac{2\pi}{\i}X\big)}{\calDC^g\big(\nu_{Y^g,Y},\tfrac{2\pi}{\i}X\big)}{\rm e}^{k(\omega+2\pi \i\pair{\mu-\xi}{X})}.
\end{gather}
\end{Lemma}
\begin{proof}As $T_I$ acts trivially on $Y$ and $\mu(Y)\subset I$, the characteristic forms in~\eqref{eqn:2} only depend on the component of $X$ (resp.~$t$) in~$\t_I^\prime$ (resp.~$T_I^\prime$). Likewise as $\xi \in (\Lambda \otimes \bQ) \cap I$, $t^{-k\xi}$ only depends on the component~$g$ of~$t$ in~$T_I^\prime$. This means the following expression can be split off from~\eqref{eqn:2} and evaluated separately:
\begin{gather}\label{eqn:SplitOff}
\sum_{h \in S_I}\int_{\t_I}{\rm d} X \sigma_h(X)\de^{\g_\sigma/\t}\big(1-h^{-1}g^{-1}{\rm e}^{-X}\big).
\end{gather}
The determinant is given by a product:
\[ \prod_{\alpha \in \R_+^{\g_\sigma}} \big(1-h^{-\alpha}g^{-\alpha}{\rm e}^{-2\pi \i \pair{\alpha}{X}}\big).\]
When the product over $\R_+^{\g_\sigma}$ is expanded, we obtain an alternating sum of terms of the form $h^{-\zeta}g^{-\zeta}{\rm e}^{-2\pi \i \pair{\zeta}{X}}$, where $\zeta$ is a sum of a subset of $\R_+^{\g_\sigma}$. The elements of $\R_+^{\g_\sigma}$ lie in $\tn{ann}(\z_\sigma)$, the annihilator of $\z_\sigma$ in $\t^\ast$. Since $\t^\ast=\z_\sigma^\ast\oplus \ann(\z_\sigma)$ and $I \subset \z_\sigma^\ast$, it follows that either $\zeta=0$ or else $\zeta \notin I$.

We claim that if $\zeta \ne 0$, then the corresponding contribution to \eqref{eqn:SplitOff} is $0$. Indeed taking the Fourier transform of the first equation in \eqref{eqn:Sum1}, we find that for any $[\zeta] \in \Lambda/(\Lambda \cap I)$, the weight lattice of $T_I$, we have
\[ \sum_{h \in S_I} h^{-[\zeta]}\int_{\t_I}\sigma_h(X){\rm e}^{-2\pi \i\pair{[\zeta]}{X}}{\rm d}X=\delta_0([\zeta]),\]
where $\delta_0$ is the function on $\Lambda/(\Lambda \cap I)$ equal to $1$ at $0$ and $0$ otherwise, obtained by Fourier transform of the constant function $1$ on $T_I$. Thus for $\zeta \in \Lambda$,
\begin{gather*}
\sum_{h \in S_I} h^{-\zeta}\int_{\t_I}\sigma_h(X){\rm e}^{-2\pi \i\pair{\zeta}{X}}{\rm d} X=\delta_{\Lambda \cap I}(\zeta),
\end{gather*}
where $\delta_{\Lambda \cap I}$ is the function on $\Lambda$ equal to $1$ on $\Lambda \cap I$ and $0$ otherwise. In particular if $\zeta \notin I$ we see that the corresponding contribution in \eqref{eqn:SplitOff} vanishes.

On the other hand, using equation \eqref{eqn:Sum1}, the contribution from $\zeta=0$ to \eqref{eqn:SplitOff} is
\[ \sum_{h \in S_I}\int_{\t_I}{\rm d} X \sigma_h(X)=1. \]
This yields the expression on the right-hand-side of \eqref{eqn:Asympt2}.
\end{proof}

\looseness=-1 We can now complete the proof of Theorem \ref{thm:main}. The fibre $P=\mu^{-1}(\xi) \subset Y$ is smooth, and the quotient $\Sigma_e:=M_\xi=P/G_\sigma=P/T_I^\prime$ is an orbifold ($T_I^\prime$ acts locally freely on~$P$). By the coisotropic embedding theorem, a neighborhood of $P$ in $Y$ is $T$-equivariantly symplectomorphic to
\[ P \times B_I \subset P \times I, \]
where $B_I$ is a small ball around $\xi$ in the subspace $I \subset \t^\ast$, the moment map $\mu$ is projection to the second factor, and the symplectic form
\[ \omega|_{P\times B_I}=\omega_\xi+{\rm d}\pair{\eta-\xi}{\theta}=\omega_\xi+\pair{{\rm d}\eta}{\theta}+\pair{\eta-\xi}{F_\theta},\]
where $\omega_\xi$ is the pullback of the symplectic form on the reduced space $M_\xi$, $\eta$ is the variable in~$B_I$, $\theta \in \Omega^1(P,\t^\ast)^T$ is a connection on $P$ with curvature $F_\theta={\rm d}\theta$, and here as well as below we have omitted pullbacks from the notation. A neighborhood of~$P^g$ in~$Y^g$ is $T$-equivariantly symplectomorphic to
\[ P^g \times B_I^g=P^g \times B_I,\]
and $T_I^\prime$ acts locally freely on $P^g$, with the quotient $\Sigma_g=P^g/T_I^\prime$ being an orbifold. On the same neighborhood we have{\samepage
\begin{gather*} \Td\big(Y^g,\tfrac{2\pi}{\i}X\big)=\pr_1^\ast \Td\big(P^g,\tfrac{2\pi}{\i}X\big),\\
\nu_{Y^g,Y}=\pr_1^\ast \nu_{P^g,P} \quad \Rightarrow \quad \calDC^g\big(\nu_{Y^g,Y},\tfrac{2\pi}{\i}X\big)=\pr_1^\ast \calDC^g\big(\nu_{P^g,P},\tfrac{2\pi}{\i}X\big).
\end{gather*}
Below we will omit $\pr_1^\ast$ from the notation.}

Take the bump function $\chi$ to have its support contained in the neighborhood of $P$ where the above local normal forms are valid. We may then integrate over $I$ instead of $B_I$, since $\chi$ vanishes outside of $P\times B_I$ by assumption. On $\supp(\chi)$,
\[ {\rm e}^{k(\omega+2\pi \i\pair{\mu-\xi}{X})}={\rm e}^{k(\omega_\xi+\pair{{\rm d}\eta}{\theta}+\pair{\eta-\xi}{F_\theta}+2\pi \i\pair{\eta-\xi}{X})}.\]
Only the top degree part of ${\rm e}^{k\pair{{\rm d}\eta}{\theta}}$ contributes to the integral over $I$; this top degree part is $(-1)^{n(n-1)/2}k^n {\rm d}\eta \cdot \Theta$, where $n=\dim(I)$, ${\rm d}\eta=\Pi  {\rm d}\eta^a$, $\Theta=\Pi \theta_a$ in terms of coordinates on~$I$. The sign $(-1)^{n(n-1)/2}$ relates the symplectic and product orientations for $P^g \times I$, so will be absorbed when we use Fubini's theorem to write the integral over $P^g\times I$ as an iterated integral. Let $\ol{\chi}(\eta)=\chi(\eta+\xi)$, a bump function on $I$ supported near $0$. Making these substitutions, as well as a change of variables $\eta \leadsto \eta+\xi$ in the integral over~$I$, the asymptotic expression~\eqref{eqn:Asympt2} for $m(k,k\xi)$ simplifies to
\begin{gather*}
k^n\sum_g g^{-k\xi} \int_{\t_I^\prime \times I} {\rm d}X\, {\rm d}\eta\,{\rm e}^{2\pi \i k\pair{\eta}{X}} \sigma_g(X)\ol{\chi}(\eta) \int_{P^g} \Theta \frac{g_L^k\Td\big(P^g,\tfrac{2\pi}{\i}X\big)}{\calDC^g\big(\nu_{P^g,P},\tfrac{2\pi}{\i}X\big)}{\rm e}^{k(\omega_\xi+\pair{\eta}{F_\theta})}.
\end{gather*}

We need the following special case of the stationary phase expansion.
\begin{Proposition}[{stationary phase expansion, cf.~\cite[Lemma 7.7.3]{Hormander1}}]
Let $u(X,\eta)$ be a Schwartz function. We have the following asymptotic expansion in $k$:
\[ \int_{\t_I^\prime \times (\t_I^\prime)^\ast}{\rm d}X\, {\rm d}\eta\, {\rm e}^{2\pi \i k\pair{\eta}{X}}u(X,\eta)\sim \frac{1}{k^n}\sum_{j=0}^\infty \frac{1}{j!} \left(\sum_a \frac{\i}{2\pi k}\frac{\partial}{\partial \eta_a}\frac{\partial}{\partial X^a}\right)^j u(0,0).\]
\end{Proposition}
\begin{Remark}
To obtain the expression here from the expression appearing in \emph{loc.\ cit.}, one sets $x=(X,\eta) \in \bR^{2n}$ and $A(X,\eta)=(\eta,X)$. Note also that in H{\"o}rmander's notation $D=-\i({\rm d}/{\rm d}x)$.
\end{Remark}

We apply this to the smooth compactly supported function
\[ u(X,\eta)=\sigma_g(X)\ol{\chi}(\eta) \int_{P^g} \Theta \frac{g_L^k\Td\big(P^g,\tfrac{2\pi}{\i}X\big)}{\calDC^g\big(\nu_{P^g,P},\tfrac{2\pi}{\i}X\big)}{\rm e}^{k(\omega_\xi+\pair{\eta}{F_\theta})}.\]
Although this function depends on $k$, the dependence is quasi-polynomial, and so the expansion still applies. Since $\sigma_g(X)$, $\ol{\chi}(\eta)$ equal $1$ in a neighborhood of~$0$, they have no effect on the expansion. The $\eta$ derivatives $k^{-1}\partial_{\eta_a}$ operate only on the factor ${\rm e}^{k\pair{\eta}{F_\theta}}$. The combined effect of the operator $\Sigma_a\,(\i/2\pi k)\partial_{\eta_a}\partial_{X^a}$ is to replace $X$ with $(\i/2\pi)F_\theta$, yielding the asymptotic expansion
\begin{gather}\label{eqn:4}
m(k,k\xi)\sim \sum_g g^{-k\xi} \int_{P^g} \Theta \frac{g_L^k\Td\big(P^g,F_\theta\big)}{\calDC^g(\nu_{P^g,P},F_\theta)}{\rm e}^{k\omega_\xi}.
\end{gather}
(By substituting $F_\theta$ for $X$ in $\Td\big(P^g,X\big)$, $\calDC^g(\nu_{P^g,P},X)^{-1}$, we mean to take the Taylor expansion around $X=0$ and substitute the differential form $F_\theta$.) At this stage we see that the contribution of $g \in S_I^\prime$ vanishes unless $P^g \ne \varnothing$, so that $S_I^\prime = S_P$ ($S_P$ is as in Theorem~\ref{thm:main}). As the characteristic forms $\Td\big(P^g,F_\theta\big)$, $\calDC^g(\nu_{P^g,P},F_\theta)$ appear multiplied by the form~$\Theta$, which has top degree in the $T_I^\prime$ orbit directions, we can replace these characteristic forms with their horizontal parts. Substituting $F_\theta$ for $X$ and taking the horizontal part is the definition of the Cartan map $\tn{Car}_\theta$ for the locally free action of $T_I^\prime$ on the space $P^g$, hence the result is the pullback along the map~$P^g \rightarrow \Sigma_g=P^g/T_I^\prime$ of the form
\[ \frac{\Td(\Sigma_g)}{\calDC^g(\nu_{\Sigma,\Sigma^g})}. \]
(See our remarks in the introduction regarding characteristic forms for orbifolds.) Similarly the $1^{\rm st}$ Chern form $c_1(L_\Sigma)$ is obtained by applying the Cartan map to the equivariant symplectic form $\omega_\t(X)=\omega-\pair{\mu}{X}$, and results in $c_1(L_\Sigma)=\omega_\xi-\pair{\xi}{F_\theta}$. Hence $\Ch(L_\Sigma)={\rm e}^{c_1(L_\Sigma)}={\rm e}^{\omega_\xi-\pair{\xi}{F_\theta}}$. The integral over the fibres of $P^g \rightarrow \Sigma_g$ then gives $1/d$, where $d \colon \Sigma=\sqcup \Sigma_g \rightarrow \bZ$ is the locally constant function giving the size of the generic stabilizer for the $T_I^\prime \simeq T/T_I$ action on $\sqcup P^g \rightarrow \Sigma$. Equation~\eqref{eqn:4} becomes
\begin{gather}\label{eqn:5}
m(k,k\xi)\sim \sum_{g\in S_P} g^{-k\xi} \int_{\Sigma_g} \frac{1}{d}\frac{g_L^k\Ch(L_\Sigma)^k\Td(\Sigma)}{\calDC^g(\nu_{\Sigma_g,\Sigma})}{\rm e}^{k\pair{\xi}{F_\theta}}.
\end{gather}
By Corollary \ref{cor:asymptoticdetermination}, $m(k,k\xi)$ is a quasi-polynomial function of $k$, hence the asymptotic expansion must be exact, or in other words, `$\sim$' in equation \eqref{eqn:5} can be replaced with `$=$'. Thus setting $\lambda=k\xi$ we have
\begin{gather}
\label{eqn:6}
m(k,\lambda)= \sum_{g\in S_P} g^{-\lambda} \int_{\Sigma_g} \frac{1}{d}\frac{g_L^k\Ch(L_\Sigma)^k\Td(\Sigma)}{\calDC^g(\nu_{\Sigma_g,\Sigma})}{\rm e}^{\pair{\lambda}{F_\theta}}.
\end{gather}
The right-hand-side of equation \eqref{eqn:6} is quasi-polynomial in $(k,\lambda)$. Hence by Corollary~\ref{cor:asymptoticdetermination}, equation~\eqref{eqn:6} holds on \emph{all} of $C_\p$ (and not only at points $(k,\lambda)$ with $\lambda=k\xi$, $\xi$ a rational, weakly regular value in the relative interior of~$\p$). This completes the proof of Theorem~\ref{thm:main}.

\subsection*{Acknowledgements}

I~thank M.~Vergne and E.~Meinrenken for helpful conversations. I thank the referees for helpful comments and suggestions that improved the article.

\pdfbookmark[1]{References}{ref}
\LastPageEnding


\begin{thebibliography}{99}
\footnotesize\itemsep=0pt

\bibitem{BerlineGetzlerVergne}
Berline N., Getzler E., Vergne M., Heat kernels and {D}irac operators,
 \textit{Grundlehren der Mathematischen Wissenschaften}, Vol.~298,
 \href{https://doi.org/10.1007/978-3-642-58088-8}{Springer-Verlag}, Berlin, 1992.

\bibitem{daSilvaThesis}
Canas~da Silva A.M.L.G., Multiplicity formulas for orbifolds, Ph.D.~Thesis,
 {M}assachusetts Institute of Technology, 1996.

\bibitem{GuilleminSternbergConjecture}
Guillemin V., Sternberg S., Geometric quantization and multiplicities of group
 representations, \href{https://doi.org/10.1007/BF01398934}{\textit{Invent. Math.}} \textbf{67} (1982), 515--538.

\bibitem{Hormander1}
H\"{o}rmander L., The analysis of linear partial differential
 operators.~{I}.~Distribution theory and Fourier ana\-ly\-sis, 2nd~ed., \textit{Grundlehren
 der Mathematischen Wissenschaften}, Vol.~256, \href{https://doi.org/10.1007/978-3-642-61497-2}{Springer-Verlag},
 Berlin, 1990.

\bibitem{ConvexitySympCuts}
Lerman E., Meinrenken E., Tolman S., Woodward C., Nonabelian convexity by
 symplectic cuts, \href{https://doi.org/10.1016/S0040-9383(97)00030-X}{\textit{Topology}} \textbf{37} (1998), 245--259.

\bibitem{VerlindeSums}
Loizides Y., Meinrenken E., The decomposition formula for {V}erlinde sums,
 \href{https://arxiv.org/abs/1803.06684}{arXiv:1803.06684}.

\bibitem{MeinrenkenAsymptotic}
Meinrenken E., On {R}iemann--{R}och formulas for multiplicities,
 \href{https://doi.org/10.1090/S0894-0347-96-00197-X}{\textit{J.~Amer. Math. Soc.}} \textbf{9} (1996), 373--389,
 \href{https://arxiv.org/abs/alg-geom/9405014}{arXiv:alg-geom/9405014}.

\bibitem{MeinrenkenSymplecticSurgery}
Meinrenken E., Symplectic surgery and the {${\rm Spin}^c$}-{D}irac operator,
 \href{https://doi.org/10.1006/aima.1997.1701}{\textit{Adv. Math.}} \textbf{134} (1998), 240--277, \href{https://arxiv.org/abs/dg-ga/9504002}{arXiv:dg-ga/9504002}.

\bibitem{MeinrenkenEncyclopedia}
Meinrenken E., Equivariant cohomology and the {C}artan model, in Encyclopedia
 of Mathematical Physics, \href{https://doi.org/10.1016/b0-12-512666-2/00344-8}{Elsevier}, 2006, 242--250.

\bibitem{MeinrenkenSjamaar}
Meinrenken E., Sjamaar R., Singular reduction and quantization,
 \href{https://doi.org/10.1016/S0040-9383(98)00012-3}{\textit{Topology}} \textbf{38} (1999), 699--762, \href{https://arxiv.org/abs/dg-ga/9707023}{arXiv:dg-ga/9707023}.

\bibitem{ParadanRiemannRoch}
Paradan P.-E., Localization of the {R}iemann--{R}och character,
 \href{https://doi.org/10.1006/jfan.2001.3825}{\textit{J.~Funct. Anal.}} \textbf{187} (2001), 442--509,
 \href{https://arxiv.org/abs/math.DG/9911024}{arXiv:math.DG/9911024}.

\bibitem{ParadanWallCrossing}
Paradan P.-E., Wall-crossing formulas in {H}amiltonian geometry, in Geometric
 Aspects of Analysis and Mechanics, \textit{Progr. Math.}, Vol.~292,
 \href{https://doi.org/10.1007/978-0-8176-8244-6_11}{Birkh\"{a}user/Springer}, New York, 2011, 295--343, \href{https://arxiv.org/abs/math.SG/0411306}{arXiv:math.SG/0411306}.

\bibitem{WittenNonAbelian}
Paradan P.-E., Vergne M., Witten non abelian localization for equivariant
 {K}-theory, and the {$[Q,R]=0$} theorem, \href{https://doi.org/10.1090/memo/1257}{\textit{Mem. Amer. Math. Soc.}}
 \textbf{261} (2019), 71~pages, \href{https://arxiv.org/abs/1504.07502}{arXiv:1504.07502}.

\bibitem{SzenesVergne2010}
Szenes A., Vergne M., {$[Q,R]=0$} and {K}ostant partition functions,
 \href{https://doi.org/10.4171/LEM/63-3/4-8}{\textit{Enseign. Math.}} \textbf{63} (2017), 471--516, \href{https://arxiv.org/abs/1006.4149}{arXiv:1006.4149}.

\bibitem{TianZhang}
Tian Y., Zhang W., An analytic proof of the geometric quantization conjecture
 of {G}uillemin--{S}ternberg, \href{https://doi.org/10.1007/s002220050223}{\textit{Invent. Math.}} \textbf{132} (1998),
 229--259.

\end{thebibliography}
\end{document}